\documentclass[a4paper,11pt]{scrartcl}
\usepackage{amsmath,amsfonts, amsthm}
\usepackage{hyperref}
\usepackage{stmaryrd}

\usepackage{todonotes}

\usepackage{graphicx,subcaption}
\usepackage{pgfplots, pgfplotstable}

\newcommand{\RR}{\mathbb{R}}
\newcommand{\CC}{\mathbb{C}}
\newcommand{\from}{\colon}
\newcommand{\SSS}{\mathbb{S}}
\newcommand{\lie}[1]{\mathfrak{#1}}
\newcommand{\lb}{\llbracket}
\newcommand{\rb}{\rrbracket}

\DeclareMathOperator{\range}{ran}
\DeclareMathOperator{\vol}{vol} % area form on R^2

\newcommand{\M}{C}

\DeclareMathOperator{\area}{area}

\newcommand{\calV}{\mathcal{V}}
\newcommand{\calH}{\mathcal{H}}
\newcommand{\calE}{\mathcal{E}}
\DeclareMathOperator{\Diff}{Diff}
\DeclareMathOperator{\Span}{Span}
\newcommand{\bigO}{\mathcal{O}}

\theoremstyle{plain}
\newtheorem{theorem}{Theorem}

\theoremstyle{definition}
\newtheorem{definition}{Definition}
\newtheorem*{corollary}{Corollary}
\newtheorem*{remark}{Remark}

\title{The Prytz connections}
\author{Geir Bogfjellmo\footnote{Corresponding author, Norwegian University of Life Sciences}, Charles Curry\footnote{Norwegian University of Science and Technology}, Sylvie Vega--Molino\footnote{University of Bergen}}

\usepackage[backend=biber, style=alphabetic]{biblatex}
\pgfplotsset{compat=1.18} 
\begin{document}
	
	\maketitle

    \abstract{
    The Prytz planimeter is a simple mechanical device that historically was used to approximate areas of plane regions.

    In this article, we present a mathematical description and analysis of the planimeter in terms of sub-Riemannian geometry and in terms of connections and horizontal lifts -- central concepts in differential geometry.
    }
	\section{Introduction and historical background}

	A \emph{planimeter} is a mechanical or electronic device for measuring the area of a region $\Omega$, typically by tracing its outline.
    The first planimeter was invented by Jakob Amsler-Laffon in 1854.  
   
    For more background on planimeters, see also Prof. Foote's webpage \url{http://persweb.wabash.edu/facstaff/footer/Planimeter/PLANIMETER.HTM}.

    The key operating principle for planimeters is the Moving Segment Theorem \cite{foote2006volume}.
	\begin{theorem} Let $p,q\from [0, T]\to \RR^2$ be two parametrized closed curves in the plane, and let $\ell(t)$ denote the moving line segment from $p(t)$ to $q(t)$. Let $A_\ell$ be the signed area swept out by $\ell(t)$, $t\in [0,T]$, and $A_p, A_q$ the signed areas of the regions enclosed by $p$ and $q$, respectively. Then
	\[A_\ell=A_p-A_q.\]
	\end{theorem}
    \begin{figure}[htp]
    \begin{center}
	\includegraphics[width=0.6\textwidth]{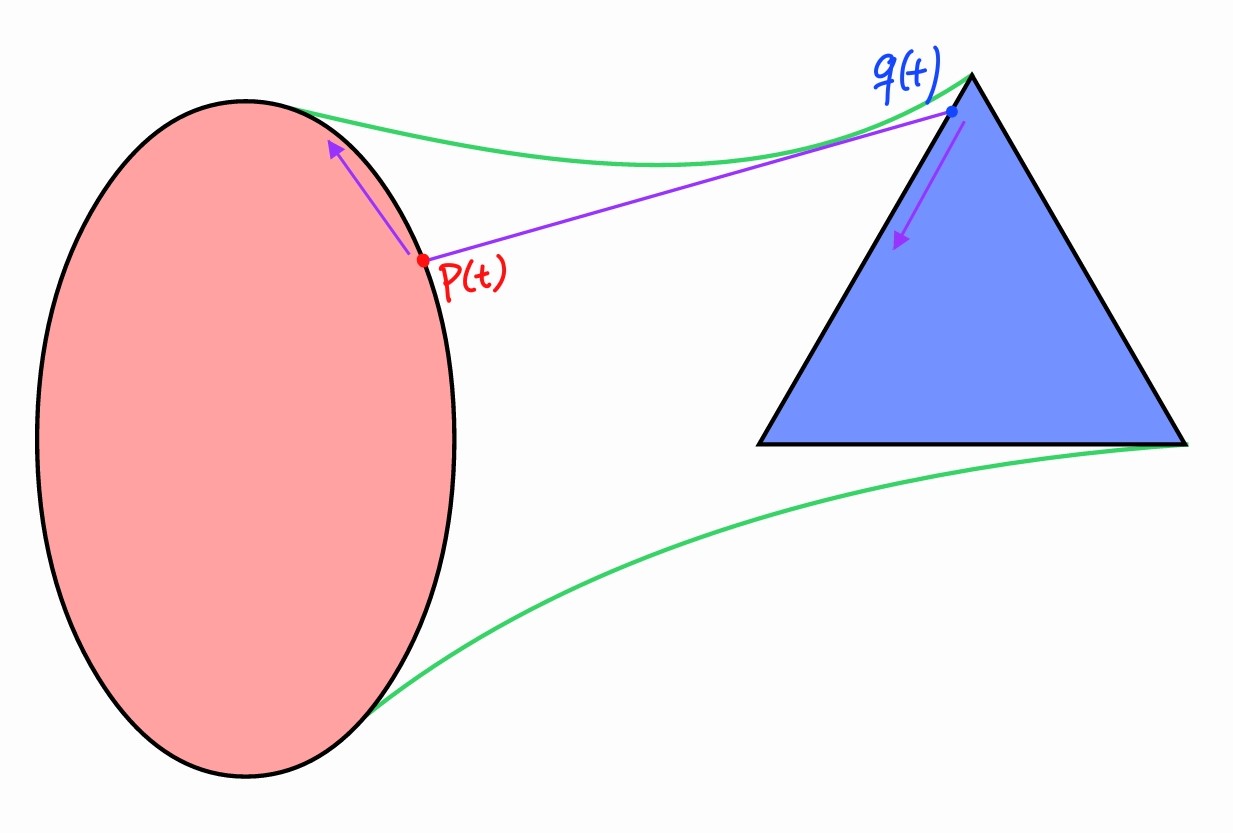}
    \end{center}
    \caption{An illustration of the moving segment theorem. The area between the two shapes is swept out twice, once in the positive direction and once in the negative direction.}
    \label{fig: mst}
    
	\end{figure}
	
 	In a planimeter, the line segment $\ell$ is realized as a rigid rod of a fixed length, and one end of the rod is made to follow the perimeter of a region, $p\from [0,T]\to \partial \Omega_1\subset \RR^2.$ The planimeter is then constructed such that such that $A_\ell$ and $A_q$ are either restricted to be zero or possible to measure.
 	
 	In the Amsler planimeters, $q$ is restricted to a curve in $\RR^2$ (a circle for the polar planimeters, a line for the linear planimeters). The two restrictions imposed on $q$ work together to ensure that $q(T)=q(0)$, and that $A_q=0$. 
    Thus the area of $A_p=\area(\Omega_1)$ can be read directly from the signed area swept by the line segment. This signed area is then measured via a gauge driven by a perpendicular wheel on the rod.
 	
 	(See \url{https://mathweb.ucsd.edu/~jeggers/Planimeter/Amsler_Radial/Amsler_Radial_gallery.html} or \cite[Section 8.4]{bryant2008round} for pictures.)
    
    The Amsler planimeters were complicated pieces of equipment, and were quite costly.
    
    A competitor to Amsler's planimeters appeared in 1875. In contrast to Amsler's planimeters, Holger Prytz' planimeter can be constructed by ``a country blacksmith'' (Prytz' own suggestion \cite{prytz1896prytz}) or by yourself from a metal coathanger (see \cite[Chapter 8]{bryant2008round} for a recipe.)
    
    The Prytz planimeter (also known as a ``stang'' or ``hatchet'' planimeter) is mathematically not as precise as the Amsler planimeter. Strictly speaking, it does not measure the area of $\Omega_1$, but a more complicated geometric quantity approximating the area.
    
    Despite its apparent simplicity, the Prytz planimeter is a geometrically interesting object. In the present article, we will show how the motion of the planimeter induces a sub-Riemannian structure on its configuration space, how the motion of the planimeter can be described as the horizontal lift of various connections in differential geometry, and how this can be used to understand what the Prytz planimeter measures and its relation to the area of $\Omega_1$.

    We conclude by noting that the planimeter is a type of kinematic linkeage of more general interest, being for example closely related to the motion of articulated vehicles.

    %\todo{Is this better? --Geir}
    
    \section{The Prytz planimeter}

    \begin{figure}
        \includegraphics[scale=0.1, angle=90]{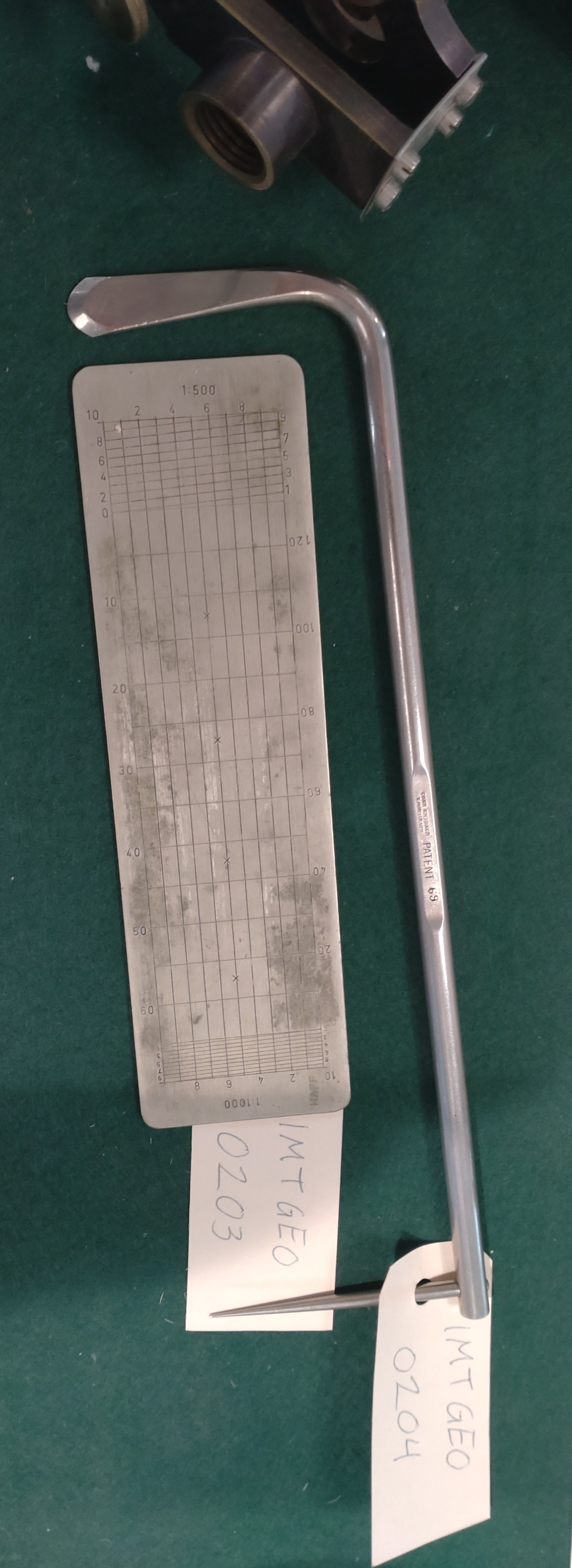}
        \caption{A Prytz planimeter}
    \end{figure}
     
   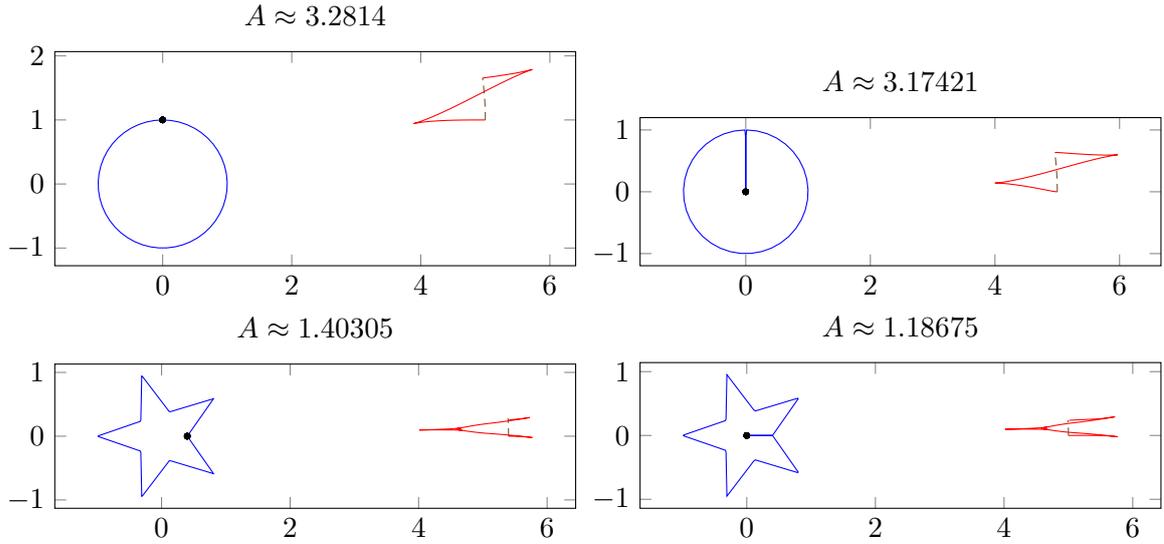
\begin{figure}
    	\begin{tikzpicture}
	\pgfplotstableread{Figuredata/Circle1.csv}\datatable
	\def \totangle {0.131255}
	\def \xstart {0}
	\def \ystart {1}
	\def \lll {5}
	\pgfmathsetmacro\Area{\lll*\lll*\totangle}
	
	\begin{axis}[axis equal image =true, 
				trig format plots=rad,
				title=$A\approx \Area$]
		\addplot table[x=px,y=py, mark=none]{\datatable};
		\addplot table[x=qx,y=qy, mark=none]{\datatable};
		\addplot+[domain=0:\totangle, densely dashed, mark=none] ({\xstart+5*cos(x)}, {\ystart+5*sin(x)});
		\addplot+[mark=*, mark size=1] (\xstart,\ystart);
	\end{axis}
\end{tikzpicture}
\begin{tikzpicture}
	\pgfplotstableread{Figuredata/Circle2.csv}\datatable
	\def \totangle {0.12697}
	\def \xstart {0}
	\def \ystart {0}			
	\def \lll {5}
	\pgfmathsetmacro\Area{\lll*\lll*\totangle}
	
	\begin{axis}[axis equal image =true, 
		trig format plots=rad,
		title=$A\approx \Area$]
		\addplot table[x=px,y=py, mark=none]{\datatable};
		\addplot table[x=qx,y=qy, mark=none]{\datatable};
		\addplot+[domain=0:\totangle, densely dashed, mark=none] ({\xstart+5*cos(x)}, {\ystart+5*sin(x)});
		\addplot+[mark=*, mark size=1] (\xstart,\ystart);
	\end{axis}
\end{tikzpicture}
\begin{tikzpicture}
	\pgfplotstableread{Figuredata/Star1.csv}\datatable
	\def \totangle {0.05612}
	\def \xstart {0.4}
	\def \ystart {0}
	\def \lll {5}
	\pgfmathsetmacro\Area{\lll*\lll*\totangle}

	\begin{axis}[axis equal image =true, 
				trig format plots=rad,
				title=$A\approx \Area$]
		\addplot table[x=px,y=py, mark=none]{\datatable};
		\addplot table[x=qx,y=qy, mark=none]{\datatable};
		\addplot+[domain=0:\totangle, densely dashed, mark=none] ({\xstart+5*cos(x)}, {\ystart+5*sin(x)});
		\addplot+[mark=*, mark size=1] (\xstart,\ystart);
	\end{axis}
\end{tikzpicture}
\begin{tikzpicture}
	\pgfplotstableread{Figuredata/Star2.csv}\datatable
	\def \totangle {0.04747}
	\def \xstart {0}
	\def \ystart {0}
	\def \lll {5}
	\pgfmathsetmacro\Area{\lll*\lll*\totangle}
	
	\begin{axis}[axis equal image =true, 
		trig format plots=rad,
		title=$A\approx \Area$]
		\addplot table[x=px,y=py, mark=none]{\datatable};
		\addplot table[x=qx,y=qy, mark=none]{\datatable};
		\addplot+[domain=0:\totangle, densely dashed, mark=none] ({\xstart+5*cos(x)}, {\ystart+5*sin(x)});
		\addplot+[mark=*, mark size=1] (\xstart,\ystart);
	\end{axis}
\end{tikzpicture}
    	
    	\caption{Some examples of planimeter paths. The blue line follows the tracer end $p(t)$ and the red line the chisel end $q(t)$. The circular arc segment between $q(0)$ and $q(T)$ with center in $p(0)=p(T)$ is shown as a dashed line. The exact values are
    		$A=\pi\approx 3.14159$ for the circle, and $A=2\sin\left(\frac{\pi}{5}\right)\approx 1.17557$ for the star.}
      \label{fig: planimeter}
    \end{figure}
    
    The Prytz planimeter is a rigid metal rod with a perpendicular prong at either end. One prong, the \emph{tracer end} is sharpened into a point. The other prong, which we will call the \emph{chisel end}, is sharpened into an edge parallel to the rod.
    
   	In operation, the tracer end $p$ is made to follow the boundary of a (simply connected) region $\Omega_1$ drawn on a paper. The motion induces the chisel end $q$ to follow a curve in the plane limited by nonholonomic constraints.
   	
   	When tracer end returns to its starting point, the chisel end will not have returned to its starting point, but will have rotated by an angle $\Delta \theta$.  See Figure \ref{fig: planimeter} for examples of paths traced out by the planimeter.

    By the moving segment theorem, it is possible to show that
    \[\area(\Omega_1) = l^2\Delta \theta  + A_q\]
    where $A_q$ is the signed area between the path of the chisel and $\Delta\theta$ and the circular arc between $q(0)$ and $q(T)$ with radius $l$. The proof can be found in \cite[Section 8.7]{bryant2008round} or \cite{foote1998geometry}.
    
   	To increase the accuracy of the measurement, i.e, reduce $A_q$, one possibility suggested by Prytz is to start with the tracer at the centroid\footnote{In practice: at an estimation of the centroid.} of $\Omega_1$, move the tracer out to the boundary along a straight line, then trace the boundary of $\Omega_1$, before returning to the start point by retracing the straight line in reverse.
   	
   	\begin{remark}
   	$\area(\Omega_1)$ can approximated as either 
   	\begin{equation}
   		\area(\Omega_1) \approx l^2\Delta \theta
   		\label{eq: areaformulaangle}
   	\end{equation}
   	or 
   	\begin{equation}
   		\area(\Omega_1)\approx l \cdot d
   		\label{eq: areaformulacord}
   	\end{equation}
   	where $d=2l\sin\left(\frac{\Delta \theta}{2}\right)= \|q(0)-q(T)\|$ is the distance between the start and end point of the chisel end.
   	
    The simplicity of measuring $d$ makes \eqref{eq: areaformulacord} preferable from a practical point of view, and a simple series expansion shows that
    $ld-l^2\Delta \theta = \mathcal{O}\left(\frac{1}{l^4}\right)$.
    In light of the other errors inherently present in the Prytz planimeter, this difference does not matter in practice.

   	This did not prevent engineers Goodman and Scott from patenting and marketing two separate ``improved'' hatch planimeters which could accurately measure \eqref{eq: areaformulaangle}.
   	This led to a sometimes heated debate in the form of letters to the magazine \emph{Engineering}. Excerpts of this debate can be found in \cite[Chapter 8.6]{bryant2008round}.
  
   	\end{remark}

	\subsection{The configuration space}
	
	Let $p\in \RR^2$ denote the tracer end and $q\in \RR^2$ the chisel end of the planimeter. 
    When the planimeter is used, $p$ and $q$ will move around in the plane restricted by a holonomic constraint $\|p-q\|=l$ and a nonholonomic constraint $ \dot{q}\parallel q-p$. 
	
	These two constraints are sufficient to ensure that for a given  path of the tracer $p\from [0,T]\to \RR^2$, and an allowed initial position $q_0$, there is a unique path of the chisel $q\from [0,T]\to \RR^2$ satisfying the constraints and $q(0)=q_0$.

    Let $C$ denote all possible configurations of the planimeter. We have the following descriptions of $C$:
    \begin{enumerate}
        \item As a submanifold of $\RR^2$:  $C=\{(p, q)\in \RR^2\times \RR^2 \mid  \|p-q\|=l\}$.
        \item As a (trivial) fibre bundle $C=\RR^2\times \SSS_1$ with base $\RR^2$ and fibres isomorphic to $\SSS_1$. We use coordinates $(x,y,\theta)$, where $(x,y)$ are the coordinates of $p$ and $\theta$ is the angle between the positive $x$-axis and the vector $q-p$. The canonical projection is 
        \[\pi \from C\to \RR^2, \qquad \pi(x,y,\theta)=(x,y).\]
    \end{enumerate}

%    We also define the projections onto the angle:
%   \[
%    \rho\from C \to \SSS_1,\qquad  \rho(x,y,\theta)=\theta,
%    \]
%    and onto the chisel end $q$:
%    \[\pi_q\from C\to \RR^2, \qquad \pi_q(x,y,\theta)=(x+l\cos \theta, y+l\sin \theta). \]

    Let $e\from[0,T]\to C, e(t)=(x(t), y(t), \theta(t))$ be a possible path for the planimeter. The nonholonomic constraint $\dot{q}(t)\parallel q(t)-p(t)$ can be written as $-\sin \theta \dot{x}+\cos \theta  \dot{y}+l\dot{\theta}=0$
    or
    $\eta_e(\dot{e})=0$, where $\eta$ is the one-form
    \begin{equation}
    \eta=-\sin \theta  dx+\cos \theta  dy +ld\theta.
    \label{eq: so2}
    \end{equation}

    \subsection{The motion of the planimeter}
    Let $p\from [0,T]\to \RR^2$ form the boundary of some area we want to measure the area of. When the tracer end is moved according along the curve $p(t)=(x(t),y(t))$, the chisel end is dragged or pushed according to the nonholonomic constraint. The resulting angle $\theta\from [0,T]\to \SSS_1$ solves the ordinary differential equation
    \begin{equation}
    \dot{\theta}= \frac{1}{l}\left(\sin \theta \dot{x}-\cos \theta \dot{y}\right),
    \label{eq: thetaODE}
    \end{equation}
    with initial value $\theta(0)=\theta_0$ given by the initial orientation of the planimeter.

    At the end, the area is approximated as $A\approx l^2(\theta(T)-\theta(0))$. This approximation is dependent on the initial angle $\theta_0$.

    \section{Prytz connections and sub-Riemannian geometry}

    The above equation describes how a curve in the plane induces a motion of a planimeter. In the fibre bundle interpretation of the configuration space, this amounts to a prescription of how to lift any curve on the base of the bundle to the total space. Identifying this with a ``horizontal lift'' we have arrived at the heart of differential geometry. The standard procedure is to define a connection, from which we obtain horizontal lifts. An alternative perspective arising from Cartan geometry leads to the closely related notion of development. It is often desirable to define connections on a principal bundle satisfying an equivariance -- a principal connection.

    We will return to a principal connection for the planimeter motions in section \ref{sec: principalPrytz}. In this section, however, we consider connections which lack this equivariance.
    %The ones considered in this section are almost of this type, but will lack the requisite equivariance, and hence determine pseudoconnections.

	\subsection{Infinitesimal connection}
    In this section, we describe how the Prytz planimeter describe a connection on the fibre bundle $C$.
    
    Recall that for a fibre bundle $C$ over the base $\RR^2$, the \emph{vertical bundle} is the vector subbundle $\calV=\ker T\pi \subset TC$.
    In our case, the vertical bundle is spanned by the vector field $\frac{\partial}{\partial \theta}$.

    For every $e\in C$, we can form the vector space $\calH_e\subset T_eC$ consisting of all tangent vectors $v$ satisfying the nonholonomic constraint $\eta_e(v)=0$. Then $\calH_e$ is a complement to $\calV_e$, and
    $\calH=\bigcup_{e\in C} \calH_e$ is a vector subbundle everywhere transversal to $\calV$ that we call the \emph{horizontal bundle}.

    Such a splitting of $TC$ into a vertical and horizontal bundle is what Ehresmann called an infinitesimal connection on a fibre bundle.
    
	\begin{definition}[\cite{marle2014works, ehresmann1950connexions}]
    \label{def: infinitesmalconnection}
		Let $E(M,F)$ denote a fibre bundle with base manifold $M$ and fibres diffeomorphic to $F$, and let $\pi\from E \to M$ be the canonical projection. An \emph{infinitesimal connection} on $E$ is a vector sub-bundle $\calH \subset TE$, that is transversal to the vertical bundle $\calV=\ker T\pi \subset TE$, i.e. such that for each $e\in E$, $T_e E=\calV_e \oplus \calH_e$.
	\end{definition}
	Equivalently, we can consider an infinitesimal connection to be defined by a smooth section of linear maps $\Phi_e\from T_eE\to T_eE$ with constant rank satisfying $\Phi_e\circ \Phi_e=\Phi_e$. The horizontal and vertical bundles are given by
    $\calH_e=\ker \Phi_e$ and $\calV_e=\range \Phi_e$.
	
    Equipped with an infinitesimal connection, we can define horizontal lifts of paths in the base manifold $\RR^2$.
    \begin{definition}
    
        Let $E(M,F)$ be a fibre bundle equipped with horizontal bundle $H$, 
        let $\gamma \from [0,T]\to M$ be a smooth curve in the base $M$ with $\gamma(0)=p$.
        and let $e$ be point in the fibre $\pi^{-1})(p).$ The horizontal lift of $\gamma$ through $e$ is the unique smooth curve $\tilde{\gamma}\from[0,T]\to E$ satisfying
        $\pi(\tilde{\gamma(t)})=\gamma(t)$, $\dot{\tilde{\gamma}}(t)\in H_{\tilde{\gamma}(t)}$ for all $t$ and $\tilde{\gamma}(0)=e$.
    \end{definition}
    We also define the horizontal lift of a vector field $U$ over $M$ to be the vector field $\tilde{U}$ over $E$ that is everywhere horizontal and $\pi$-related to $U$.
    
    For the Prytz planimeter, the horizontal lift of a path $p\from[0,T]\to \RR$ describes the motion of the planimeter when the tracer end follows the path $p(t)$.

    \begin{remark}
    To get a first understanding of why the Prytz planimeter measures area, consider what happens when the tracer follows the boundary of an infinitesimal square:
    \begin{center}
    \begin{tikzpicture}
      % Define rectangle coordinates
      \coordinate (A) at (0,0);
      %\node[below left] at (A) {$(x,y)$};
      \coordinate (B) at (2,0);
      %\node[below right] at (B) {$(x+\epsilon, y)$};
      \coordinate (C) at (2,2);
      %\node[above right] at (C) {$(x+\epsilon, y+\epsilon)$};
      \coordinate (D) at (0,2);
      %\node[above left] at (D) {$(x,y+\epsilon)$};
    
      % Draw arrows
      \draw[->] (A) --node[below]{$\epsilon \frac{\partial}{\partial x}$} (B);
      \draw[->] (B) --node[right]{$\epsilon \frac{\partial}{\partial y}$}  (C);
      \draw[->] (C) --node[above]{$-\epsilon \frac{\partial}{\partial x}$}  (D);
      \draw[->] (D) --node[left]{$-\epsilon \frac{\partial}{\partial y}$}  (A);
    \end{tikzpicture}
    \end{center}
    
    In the picture, $\frac{\partial}{\partial x}$ and $\frac{\partial}{\partial y}$ denote the coordinate vector fields on $\RR^2$.

    The horizontal lift will follow the horizontally lifted vector fields
    \begin{equation}\begin{aligned}
    \widetilde{\frac{\partial}{\partial x}}=X&=\frac{\partial}{\partial x} + \frac{1}{l}\sin \theta \frac{\partial}{\partial \theta}\\
	\widetilde{\frac{\partial}{\partial y}}=Y&=\frac{\partial}{\partial y} - \frac{1}{l}\cos \theta \frac{\partial}{\partial \theta}
    \end{aligned}
    \label{eq: horizontalvectorfields}
    \end{equation}
    In contrast to the original path, the lifted path is not closed. The failure of such an infinitesimal square to be closed is measured by the \emph{Jacobi--Lie} bracket of vector fields.
    \begin{center}
     \begin{tikzpicture}
      % Define rectangle coordinates
      \coordinate (A) at (0,0);
      \coordinate (B) at (2,0);
      \coordinate (C) at (2.3,2);
      \coordinate (D) at (0.1,2.4);
      \coordinate (E) at (-0.4, 0.6);
    
      % Draw arrows
      \draw[->] (A) --node[below]{$\epsilon X$} (B);
      \draw[->] (B) --node[right]{$\epsilon Y$}  (C);
      \draw[->] (C) --node[above]{$-\epsilon X$}  (D);
      \draw[->] (D) --node[left]{$-\epsilon Y$}  (E);
      \draw[->, dashed] (A) --node[left]{$\epsilon^2 \lb X,Y\rb $} (E);
    \end{tikzpicture}
    \end{center}
    For the vector fields \eqref{eq: horizontalvectorfields}, we have $\lb X, Y\rb= \frac{1}{l^2}\frac{\partial}{\partial \theta}$.

    Loosely speaking: 

    \emph{When the tracer traces out the boundary of an infinitesimal region with area $\epsilon^2$, the planimeter rotates by the angle $\epsilon^2/l^2$.}
    
    \end{remark}

    More precisely, the curvature of an infinitesimal connection, as defined in \cite[p. 73]{kolar1993natural}, is a two-form $R$ on $C$ taking values in $TC$.

    On arbitrary vector fields $U$ and $V$ over $C$, $R$ is given by the function
	\[R(U,V) = \Phi \left(\lb U-\Phi(U), V-\Phi(V)\rb\right)\]
    where $\Phi\from TC\to TC$ is the section of linear maps with $\ker \Phi=\calH$, $\range \Phi=\calV$, and $\lb\cdot, \cdot \rb$ is the Jacobi--Lie bracket of vector fields.

    In our case, this tracks the rotation of the planimeter
    \[R(U,W) = \pi^{*}\vol(U,W)\cdot \frac{1}{l^2}  \frac{\partial}{\partial \theta},
	\]
	where $\pi^{*}\vol$ is the pull-back of the area form on $\RR^2$ to a two-form on $C$.

    In other words, for two vectors $U_e$ and $W_e$ in $T_eC$, the curvature $R_e(U_e,W_e)$ is a vertical vector with length proportional to the area of the parallelogram spanned by the projection of the two vectors onto $T\RR^2$.

    \subsection{Sub-Riemannian perspective}
In the setting of sub-Riemannian geometry, one works with a smooth manifold $\M$ equipped with a pair $(\calE,g_\calE)$ where  $\calE$ is a subbundle of $T\M$ and $g_\calE$ is a symmetric, positive-definite (2,0)-tensor on $\calE$.  One insists that the bracket-generating condition holds, which is said to be satisfied if at every point $p \in \M$ one can generate all of $T_p\M$ by taking sufficiently many Lie brackets of vector fields in $\Gamma(\calE)$ at $p$.  Notably, this is equivalent to H\"ormander's condition in PDEs. One can define the sub-Riemannian (Carnot-Carathéodory) distance $d_{cc}(p,q)$ for two points $p,q \in \M$ by the usual infimum formula taken over the space of smooth paths $C_{\calE,p,q}$  connecting $p$ to $q$ such that $\dot\gamma(t) \in \calE_{\gamma(t)}$ at almost every point along $\gamma$.  The now famous theorem of Chow and Rashevsky \cite{C39,R38} tells us that $(\M,d_{cc})$ is a complete metric space precisely when $\calE$ is bracket-generating.  Such a triple $(\M,\calE,g_\calE)$ is called a sub-Riemannian manifold.

We introduce a sub-Riemannian structure on the description of the Prytz planimeter by defining $\calE$ to be the span of the vector fields, just as in \eqref{eq: horizontalvectorfields}
\begin{align*}
X &= \frac{\partial}{\partial x} + \frac{1}{l}\sin \theta \frac{\partial}{\partial \theta}, \\ 
Y& = \frac{\partial}{\partial y} - \frac{1}{l}\cos \theta \frac{\partial}{\partial \theta},
\end{align*}
and letting the sub-Riemannian metric $g_\calE$ be such that $X,Y$ are orthonormal.  The first bracket is $[X,Y] = \frac{1}{l^2} \frac{\partial}{\partial \theta}$ from which we see that $\calE$ is indeed bracket-generating.

\begin{corollary}
There is a path between any two configurations of a planimeter via planimeter motions.
\end{corollary}
\begin{proof}
This follows from Chow-Rashevskii.
\end{proof}

\begin{remark}
Observe that the sub-Riemannian horizontal distribution $\calE$ and the horizontal sub-bundle $\calH$ determining the Ehresmann connection in the previous section coincide.  This is to say that the connection is adapted to the sub-Riemannian structure, see \cite{Vega-Molino_2020} for an overview.
\end{remark}

We also observe that the form $\eta$ in \eqref{eq: so2} is contact, and the sub-Riemannian structure we have defined is the associated contact sub-Riemannian structure.  It follows that there can be no abnormal geodesics. 

\subsubsection*{Hamiltonian Mechanics}
Following the Hamiltonian perspective (c.f. \cite{M02,ABB20}) it is of interest to consider the Hamiltonian 
\begin{align*}
H 	&= \frac{1}{2} \left( P_X^2 + P_Y^2 \right) \\
	&= \frac{1}{2} \left( p_x^2 + p_y^2 + \frac{1}{l^2}p_\theta^2 \right) + \frac{1}{l}\left( \sin\theta p_x - \cos\theta p_y \right) p_\theta
\end{align*}

which then induces the Hamiltonian system
\begin{equation*}
\begin{aligned}
    \dot x 		&= p_x + \frac{1}{l} \sin\theta p_\theta 		
    				& \dot p_x &= 0 \\
    \dot y 		&= p_y - \frac{1}{l} \cos\theta p_\theta 		
    				& \dot p_y &= 0 \\
    \dot\theta 	&= \frac{1}{l^2} p_\theta + \frac{1}{l} \left( \sin\theta p_x - \cos\theta p_y \right)
    				& \dot p_\theta &= -\frac{1}{l}\left( \cos\theta p_x + \sin\theta p_y \right) p_\theta
\end{aligned}
\end{equation*}
the solutions of which are the normal sub-Riemannian geodesics.  

The problem reduces to solving only the last line of the system, and moreover it can be shown that $\theta$ is the solution of the autonomous differential equation
\begin{equation*}
    \ddot\theta = \frac{1}{2l^2}\left(2p_xp_y \cos(2\theta) + (p_x^2 - p_y^2)\sin(2\theta) \right).
\end{equation*} 

To give some intution, recall that a sub-Riemannian geodesic is a curve which is locally length-minimizing (for the sub-Riemannian metric).  For example, when tracing a planimeter around a closed loop the initial and final points in $C$ differ only in the $\theta$ coordinate; a sub-Riemannian geodesic projecting onto a closed loop therefore minimizes $\Delta\theta$ and so by \eqref{eq: areaformulaangle} it also approximates a minimization of area. Therefore the projection must approximate a circle.  The precise shape is determined by the higher-order error terms.

\begin{figure}
    \centering
    \begin{tikzpicture}
	\pgfplotstableread{Figuredata/Geodesic1.csv}\datatable
	\begin{axis}[axis equal image =true, legend pos=outer north east]
		\addplot+[mark=none] table[x=px,y=py]{\datatable};
        \addlegendentry{$p(t)$};
		\addplot+[mark=none] table[x=qx,y=qy]{\datatable};
        \addlegendentry{$q(t)$};
	\end{axis}
\end{tikzpicture}

\begin{tikzpicture}
	\pgfplotstableread{Figuredata/Geodesic2.csv}\datatable
	\begin{axis}[axis equal image =true]
		\addplot table[x=px,y=py, mark=none]{\datatable};
		\addplot table[x=qx,y=qy, mark=none]{\datatable};
	\end{axis}
\end{tikzpicture}
    \caption{Examples of sub-Riemannian geodesics.}
\end{figure}
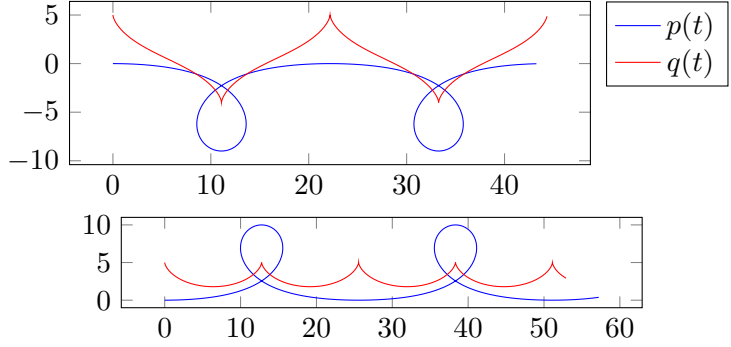

\begin{remark}
    We note that the vertical vector field $Z = \frac{1}{l^2}\frac{\partial}{\partial\theta}$ determines a foliation of $C$, however the Lie derivatives $(\mathcal{L}_U \tilde g_\calE)(Z,Z)$ do not vanish for $U \in \calE$ and so the foliation is not totally-geodesic (here $\tilde g_\calE$ is the Riemannian extention of $g_\calE$ making $X,Y,Z$ an orthonormal frame).  As a consequence the Eulerian approach to sub-Riemannian geometry via penalty metrics does not define an H-type foliation (see \cite{BGRV18}). Equivalently, we can understand that the contact structure is not $K$-contact.
\end{remark}

 \subsection{Linear and affine pseudoconnections}

    We move a step closer to mainstream differential geometry by noting that the fibres of the bundle $C$ can be identified with $SO(2)$. It is a principal bundle, where moreover the fibres are a subgroup of the general linear group, with its natural representation on the base $\RR^2$. The connection form (\ref{eq: so2}) is $\mathfrak{so}(2)$-valued, and can therefore be considered a linear pseudoconnection, i.e. the standard representation of $SO(2)$ on $\RR^2$ induces an associated connection on the tangent bundle of $\RR^2$. The $\mathrm{Ad}\,SO(2)$-equivariance is lacking however, hence the name pseudoconnection.
  
	More interesting perhaps is to take the perspective of Cartan geometry and enrich the connection form with a solder form to obtain a $\mathfrak{se}(2)$-valued pseudoconnection. This can be done in such a way that the curve traced out by the chisel end of the planimeter is the development of the curve traced out by the tracer. For this purpose, note that the chisel has coordinates $q=(\tilde{x}, \tilde{y}) = (x + l \cos (\theta), y + l\sin (\theta))$. Differentiating and rearranging, we obtain
\begin{eqnarray*}
  \dot{\tilde{x}} & = & \dot{x} \cos^2 \theta + \dot{y} \sin \theta \cos
  \theta,\\
  \dot{\tilde{y}} & = & \dot{x} \sin \theta \cos \theta + \dot{y} \sin^2
  \theta.
\end{eqnarray*}

In other words, if we define a basis of $\mathfrak{se} (2)$ by
\[
    e_1=\begin{pmatrix}
        0 & 0 & 0\\
        1 & 0 & 0\\
        0 & 0 & 0
    \end{pmatrix},
    \quad 
    e_2=\begin{pmatrix}
        0 & 0 & 0\\
        0 & 0 & 0\\
        1 & 0 & 0
    \end{pmatrix}, 
    \quad
    e_3=\begin{pmatrix}
        0 & 0 & 0\\
        0 & 0 & -1\\
        0 & 1 & 0
    \end{pmatrix},
\]
so that $e_3$ is the rotational basis vector and $e_1$ and $e_2$ the translational basis vectors, we can
define an affine pseudoconnection $\tilde{\omega}\from T\RR^2\to \lie{se}(2)$ by
\[ 
\begin{aligned}
\tilde{\omega} =&e_1\left(l \cos \theta (\cos \theta dx + \sin \theta dy) - l\sin \theta d \theta\right)\\
&+e_2\left(l \sin \theta (\cos \theta dx + \sin \theta dy) + l\cos \theta d \theta\right)\\
&+e_3\left(  l d \theta + \sin \theta dx - \cos \theta dy\right).
\end{aligned}
%=\begin{pmatrix}
%    l \cos \theta (\cos \theta dx + \sin \theta dy) - l\sin \theta d \theta\\
%     l \sin \theta (\cos \theta dx + \sin \theta dy) + l\cos \theta d \theta\\
%     l d \theta + \sin \theta dx - \cos \theta dy
\]
%   \end{pmatrix}. \]

\section{The principal Prytz connection}
\label{sec: principalPrytz}

    To obtain an equivariant connection on a principal bundle we must deal with the dependency of the subsequent motion on the initial angle.
    
	\subsection{Principal connections}
    One approach, following Foote is to look at the subgroup of diffeomorphisms on $\SSS_1$ generated by the planimeter. 
    We refer to \cite{foote1998geometry} for further details.

    First, let us recall \emph{principal bundles} and \emph{principal connections} (See \cite{kobayashi1963foundations}.)

    \begin{definition}A \emph{principal bundle} $P(M,G)$ is a fibre bundle with fibres diffeomorphic to a Lie group $G$, equipped with a right group action $P\times G\to P$. We will write the right action $(p,g)\mapsto R_g(p)$.
    \end{definition}
    Differentiating the right action with respect to $g$ at the identity $I\in G$, we get a Lie algebra morphism $\lie{g}\to \cal V$, where $\lie{g}$ is the Lie algebra of $G$ and $\cal V\subset TP$ is the vertical subbundle.
    We write this Lie algebra morphism $\xi \mapsto X_\xi$.
    \begin{definition}
        A \emph{principal connection} on $P$ is a $\lie{g}$-valued one-form $\omega$ on $P$ satisfying
        \[
        \begin{aligned}
        \omega(X_\xi)&=\xi,  &\text{for all }\xi &\in \lie{g},\\
        R_{g}^*\omega &= Ad_{g^{-1}} \circ \omega, &\text{for all } g&\in G.
        \end{aligned}
        \]
    \end{definition}
    A principal connection $\omega$ defines an infinitesimal connection in the sense of Definition \ref{def: infinitesmalconnection} via the projection $\Phi\from TP\to TP$, $\Phi(\cdot)=X_{\omega(\cdot)}.$

    In the case of a trivial principal bundle, $P=M\times G$, a principal connection $\omega$ induces a $\lie{g}$-valued form $\varpi$ on $M$ via the pullback of $\omega$ to the trivial section $M\simeq M\times \{I\} \subset P$.
    The principal connection $\omega\from TP\to \lie{g}$ is in turn uniquely defined from $\varpi\from TM\to \lie{g}$ (See \cite[Proposition II.1.4]{kobayashi1963foundations}).

    Given a connection $\omega$ on a principal bundle $P$ and a curve $\gamma(t)$ on the base $M$, one can constructs a horizontal lift as follows \cite[II.3]{kobayashi1963foundations}:
    \begin{enumerate}
        \item Take an arbitrary $C^1$ lift $v(t)\in P$ of $\gamma(t)$ to the principal bundle.
        \item Solve the equation of Lie type 
        $$ 
        \dot{\Gamma}(t) = -\omega\big(\dot{v}(t)\big) \Gamma(t),\quad \Gamma(0)=I.
        $$
        \item The horizontal lift is $u(t)=v(t)\Gamma(t)$.
    \end{enumerate} 
    In the case of a trivial principal bundle, we can take the curve $v(t)=(\gamma(t),I)$, so the equation to be solved becomes
    \begin{equation}
    \dot{\Gamma}(t) = -\xi(t) \Gamma(t), \quad \xi(t)= \varpi(\dot{\gamma}(t)).
    \label{eq: evolution}
    \end{equation}

    \subsection{The principal bundle}
    We consider $C=\RR^2\times \SSS_1$ as a trivial fibre bundle with coordinates $(x,y,\theta)$. 
    Recall the ODE \eqref{eq: thetaODE}
    \[\dot{\theta}= \frac{1}{l}\sin \theta \dot{x}-\frac{1}{l}\cos \theta \dot{y}.\]
    We can view this as a linear function from the tangent vector $(\dot{x}, \dot{y})\in T_{p}\RR^2$ to the Lie algebra of smooth vector fields on $\SSS_1$.

    By letting the initial value $\theta(0)=\theta_0$ vary, the ODE induces a family of flow maps $\Gamma(t)\in \Diff(\SSS_1)$ via $\Gamma(t)(\theta_0)=\theta(t).$

    We want to view the flow maps as the horizontal lifts of the curve $\gamma(t)=(x(t), y(t))$ in $\RR^2$. We can avoid dealing with the infinite-dimensional Lie group of diffeomorphisms and instead work with a finite-dimensional subgroup of $\Diff(\SSS_1)$.

    Let $G\subset \Diff(\SSS_1)$ be the subgroup of diffeomorphisms that can be written as flow maps of the ODE \eqref{eq: thetaODE}. Then $G$ is the group generated by the two vector fields on $\SSS_1$:
	\[\begin{cases} 
		X=\frac{1}{l}\sin \theta \frac{\partial}{\partial \theta}, \\
		Y=-\frac{1}{l}\cos \theta \frac{\partial}{\partial \theta}.
	\end{cases}
	\]

    The Lie algebra of $G$ is generated by the same vector fields. We first compute the Lie bracket of $X$ and $Y$, which we define as\footnote{The sign is due to equation \eqref{eq: evolution}: it is convenient to identify $X$ and $Y$ with right-invariant vector fields on $G$ instead of the usual left-invariant vector fields.}
    $[X,Y]=-\lb X, Y\rb.$
    \[[X,Y]=-\frac{1}{l^2}\frac{\partial}{\partial \theta}.\]
    The vector fields $X,Y, [X,Y]$ are in involution and form a three-dimensional simple Lie algebra, $\lie{g} = \Span\{X, Y, [X,Y]\}$.
 
    It can be shown that
	\[\lie{g}\simeq \lie{su}(1,1)=
	\left\{\begin{pmatrix} i\gamma & \beta \\
		\beta^{*} & -i\gamma 
	\end{pmatrix} \middle| \gamma\in \RR, \beta \in \CC \right\},
	\]
	with isomorphism given by
	\begin{equation}
	\begin{aligned}
		e_1= \begin{pmatrix} 0 & 1 \\
			1 & 0
		\end{pmatrix}& \mapsto -2\sin \theta \frac{\partial}{\partial \theta},\\
		e_2 =  \begin{pmatrix} 0 & i \\
			-i & 0
		\end{pmatrix}& \mapsto 2\cos \theta \frac{\partial}{\partial \theta}, \\
		e_3 = \begin{pmatrix} i & 0 \\
			0 & -i
		\end{pmatrix}& \mapsto 2 \frac{\partial}{\partial \theta}.
	\end{aligned}
	\label{eq: isomorphisms}
	\end{equation}
	It can also be shown (See \cite{foote1998geometry}) that the corresponding Lie group $G$ is isomorphic to the projective special unitary group of signature $(1,1)$.
	\[G\simeq PSU(1,1)= \left\{
	\begin{pmatrix} a & b\\
		b^* & a^*
	\end{pmatrix}
	\middle| |a|^2-|b|^2=1
	\right\}\Large{/} \left\{\pm I\right\},\]
	and that the corresponding diffeomorphisms on $\SSS_1$ are given by:
	\begin{equation}
		\begin{pmatrix} a & b\\
		b^* & a^*
	\end{pmatrix}(e^{i\theta})= \frac{ae^{i\theta}+b}{b^{*}e^{i\theta}+a^*},
	\label{eq: groupaction}
	\end{equation}
    where we have identified $\SSS_1$ with the unit circle in $\CC$. %% We can note that the diffeomorphisms are Möbius transforms.

    \subsection{The connection}
	We now define a principal connection on the trivial principal bundle $P=\RR^2\times G$.
 
    The relation between the principal connection and the planimeter is as follows:
    Let $\gamma\from[0,T]\to \RR^2$ be a curve in $\RR^2$ that the tracer end follows. One horizontal lift of $\gamma$ is the curve $\gamma, \Gamma \from [0,T] \to \RR^2\times G$ with $\Gamma(0)=I$ and
    \[\Gamma(t)= \begin{pmatrix} a(t) & b(t)\\ b(t)^* & a(t)^*\end{pmatrix}  \in G.
    \]
    If $\theta(0)=\theta_0$ denotes the  initial angle of the planimeter, then $\theta(t)$ is given by
    \[e^{i\theta(t)}= \Gamma(t)(e^{i\theta_0})=\frac{a(t)e^{i\theta_0}+b(t)}{b^*(t)e^{i\theta_0}+a^*(t)}.
    \]

	As a $\lie{g}$-valued connection on $\RR^2$, we can write the connection as $\varpi = \frac{1}{2l}(e_1 dx+ e_2 dy)$ or
	\[
	\varpi(v)= \frac{1}{2l}\begin{pmatrix} 0 & \kappa(v) \\
		\kappa(v)^* & 0
		\end{pmatrix}\in \lie{g}
	\]
	where $\kappa\from T_p \RR^2 \simeq  \RR^2 \to \CC$ is the standard identification. The corresponding principal connection on $P=\RR^2\times G$ satisfies $\omega(X_\xi)=\xi$ for left-invariant vector fields $X_{\xi}$ and $R_g^{*}\omega = Ad_{g^{-1}}\circ \omega$.

    Explicitly, if $(p, g)\in \RR^2\times G =P$ is a point on the principal bundle, we can write a tangent vector as $(v, X_\xi)\in T_p\RR^2\times T_g G \simeq T_{(p,g)}P$. Then
    \[\omega_{(p,g)}(v, X_\xi)= \xi+ Ad_{g^{-1}} \varpi(v).\]

    On the trivial bundle, the curvature 2-form of $\varpi$ can be computed as
	\[\bar{\Omega}_p(u,v)= d \varpi(u,v)+ [\varpi(u), \varpi(v)] =- \frac{\vol(u,v)}{2l^2}\begin{pmatrix} i & 0 \\
		0 & -i
	\end{pmatrix}, \]
    where $u, v \in T_p \RR^2$ and $\vol$ is the area form on $\RR^2$.

    This also determines the curvature of the principal connection $\omega$ (see \cite[II.5]{kobayashi1963foundations}, noting that the curvature is a tensorial 2-form of type $(\mathrm{Ad},\mathfrak{g})$)
    \[\Omega_{(p,g)}\left((u, X_\xi), (v, X_\eta)\right) = Ad_{g^{-1}}\circ \bar{\Omega}_p(u,v)
    \]
    where $(p,g)\in P=\RR^2\times G$ and $(u, X_\xi), (v, X_\eta)$ are two vectors in $T_{(p,g)}P=T_p \RR^2\times T_g G$.

	%\subsection{Cartan connection}
	
	%A possible Cartan connection is to take the principal bundle $\mathrm{Diff}(SO(2),SE(2))$, with $PSU(1,1)$ as a subbundle. 
	
	\section{Horizontal lifts and the Magnus expansion}
    We have now defined a connection on a principal bundle such that the relationship between the tracer curve and the motion of the planimeter may be understood as a horizontal lift. The key observation is that the related notions of horizontal lifts and developments typically amount to solving differential equations in a Lie group. For this purpose it is profitable to employ techniques of Lie group integration.

    It should be noted that when the tracer end follows a closed curve, the horizontal lift to $P$ describes an element of the \emph{holonomy group} of the connection. There exists a non-abelian version of the Stokes theorem \cite[Corollary 3.6]{schreiber2011smooth} that links the holonomy and curvature of a principal connection. We will here take a more pedestrian approach using the Magnus expansion.

	\subsection{Magnus Expansion}
	Let $\gamma \from [0,T]\to \RR^2$ be the curve traced out by the tracer end, i.e., a parametrization of a boundary of the region $\Omega_1$. For simplicity, assume that the curve is closed and beginning at the origin, so $\gamma(0)=\gamma(T)=(0,0)$.

    The horizontal lift to $P=\RR^2\times G$ is given by the solution to 
    \eqref{eq: evolution}.
	
	The Magnus expansion (see \cite{magnus1954exponential, iserles1999solution}) is an expansion of the evolution as the exponential of a series in $\lie{g}$,
	\[\Gamma(T) = \exp(U_1 + U_2 + U_3 +U_4\dotsc)
	\]
	where $U_1,U_2,\dotsc$ are given as integrals of Lie polynomials.
	
	In our case:
	\[
	\begin{aligned}
	U_1 &= -\int_{0}^T \xi(t) dt \\
	          &= -\frac{1}{2l} \oint_\gamma \left(dx e_1 + dy e_2\right)=0\\
	U_2 &= \frac{1}{2}\int_{0}^T\int_0^{t_1} [\xi(t_1),\xi(t_2)] dt_2 dt_1 \\
	          &= \frac{1}{8l^2} \int_0^T \left[ \dot{x}(t_1) e_1+\dot{y}(t_1) e_2, \int_0^{t_1} \dot{x}(t_2) e_1+\dot{y}(t_2) e_2 dt_2\right] d t_1\\
	          &= \frac{1}{8l^2}  \int_0^T \left[ \dot{x}(t_1) e_1+\dot{y}(t_1) e_2, x(t_1)e_1+y(t_1) e_2\right] d t_1\\
	          &= \frac{1}{4l^2} \oint_\gamma \left( xdy-ydx\right) e_3\\
	          &= \frac{A}{2l^2} \cdot e_3,
	\end{aligned}
	\]
	where $A= \area(\Omega_1)$.
	This again shows how the Prytz planimeter measures area: If we truncate the Magnus series after the leading term, we have
	\[\begin{aligned}
		\Gamma_T &=\exp(U_1+U_2)+\bigO\left(\frac{1}{l^3}\right) \\
		            &= \exp\left(\frac{A}{2l^2} e_3\right)+\bigO\left(\frac{1}{l^3}\right) \\
		            &= \begin{bmatrix} e^{i\frac{A}{2l^2}}&  0 \\ 0 &  e^{-i\frac{A}{2l^2}}\end{bmatrix}+\bigO\left(\frac{1}{l^3}\right).
	\end{aligned}\]
	The corresponding group action from \eqref{eq: groupaction} is  
	\[\Gamma_T\cdot e^{i\theta}= e^{i(\theta+\Delta \theta)}, \quad 
	\text{where}\quad \Delta \theta= \frac{A}{l^2}+\bigO\left(\frac{1}{l^3}\right).\]
    By rearrangement, we get
    \[A=l^2\Delta \theta+ \bigO\left(\frac{1}{l}\right), \]
    which shows that the ``Prytz area'' and actual area agree to an error term of order $\bigO\left(\frac{1}{l}\right)$
 
	Further computations show that the leading error terms in the ``Prytz area'' are controlled by the Magnus expansion terms:
	\[
	\begin{aligned}
	U_3&=\frac{1}{2l^3}\left(M_y e_1-M_x e_2\right),\\
	U_4&=\frac{1}{4l^4}M_2 e_3,
	\end{aligned}
	\]
	where 
    \[\begin{aligned}
         M_x&=\int_{\Omega_1} x dx\wedge dy\\
         M_y &=\int_{\Omega_1} y dx \wedge dy\\
         M_2&=\int_{\Omega_1} (x^2+y^2)dx \wedge dy
    \end{aligned}
    \]
    are first and second moments of area of the region $\Omega_1$.
	
	The term $U_3=\frac{1}{2l^3}\left(M_y e_1-M_x e_2\right)$ explains why starting in the centroid is an advantage: If $M_x=M_y=0$, the first error term is eliminated.
	
	\section{Postscript: development and trailers}
	The relationship between front and back end of the planimeter is of broader interest than simply measuring areas. We begin by noting that the curve induced in the back end by a straight line was already considered Huygens and Leibniz in the 17th century \cite[The tractrix, p. 135]{hairer2000analysis}

    Moreover, note that a bicycle is essentially a planimeter of length $l$ equal to the distance between the centres of the front and back wheel. A car can be approximated by the same construction. In practice, the radius of curvature of the tracer curve $\gamma(t)$ will be large compared to the length $l$, so that front and back follow each other closely. 

    What is of more interest is vehicles with trailers. A vehicle with a system of $n$ trailers can be approximated by $n+1$ planimeters, chained together (see \cite{ljungqvist2019motion}, note that the planimeter is here a ``kinematic bicycle''). In this case the Cartan picture of development becomes natural: the configuration space of each planimeter is isomorphic to $SE(2)$, and the motion of the ($m+1$)th planimeter is the development the the $m$th. It is hoped that this perspective can be usefully applied to control problems within trailer system, an issue of increasing practical importance due to the increased automation of vehicles and requirement to improve efficiency of road haulage \cite{ljungqvist2019motion}.
	
	\printbibliography
\end{document}